\documentclass{birkjour}
 \newtheorem{thm}{Theorem}[section]
 \newtheorem{cor}[thm]{Corollary}
 \newtheorem{lem}[thm]{Lemma}
 
 \theoremstyle{definition}
 \newtheorem{defn}[thm]{Definition}
 \newtheorem{rem}[thm]{Remark}
 \newtheorem{ex}[thm]{Example}
 \numberwithin{equation}{section}

\usepackage{amsfonts,amsmath,amssymb}

\usepackage[colorlinks]{hyperref}

\usepackage{enumerate}

\usepackage{tikz}

\usepackage{tikz-cd}

\numberwithin{equation}{section}

\newcommand{\cl}[1]{\mathcal{#1}} 

\newcommand{\pr}[1]{\mathrm{pr}(#1)} 

\newcommand{\nor}[1]{\left\Vert #1\right\Vert}

\newcommand{\Map}[1]{\mathrm{Map}(#1)} 

\newcommand{\Lat}[1]{\mathrm{Lat}(#1)} 

\newcommand{\Alg}[1]{\mathrm{Alg}(#1)} 

\begin{document}

\title[Morita equivalence of w*-rigged modules]
 {Morita equivalence of w*-rigged modules}

\author{E.Papapetros}
\address{University of Patras\\Faculty of Sciences\\ Department of Mathematics\\265 00 Patras Greece}
\email{e.papapetros@upatras.gr}


\begin{abstract}  
The w*-rigged modules over dual operator algebras were introduced by Blecher and Kashyap as a generalization of W*-modules. In this paper, we introduce two new types of Morita equivalence between right w*-rigged modules over unital dual operator algebras and we examine whether these notions imply stable isomorphism between the corresponding modules. Furthermore we investigate them in detail for the class of right w*-rigged modules over nest algebras, a class which was characterized by G.K. Eleftherakis.
\end{abstract}


\maketitle

\section{Introduction}
The notion of a Hilbert $C^{*}$-module was developed in the early 1970's by Rieffel and Paschke \cite{Paschke, Rieffel}. These modules are useful tools in the Morita equivalence of $C^{*}$-algebras. The dual version is the notion of a W*-module, that is a Hilbert $C^{*}$-module over a von Neumann algebra, which satisfies an analogue of the Riesz representation theorem for Hilbert spaces. Recently in \cite{Ble-Kas},  Blecher and Kashyap generalized the notion of W*-module to the setting of dual operator algebras. The modules introduced there are called w*-rigged modules. This work was continued by Blecher and Kraus in \cite{Ble-Kr}. Unlike the W*-module situation, w*-rigged modules do not necessarily give rise to a w*-Morita equivalence of dual operator algebras in the sense of \cite{Ble-Kas-2}.

In this paper, we introduce two new types of Morita equivalence between right w*-rigged modules over unital dual operator algebras. We call them Morita equivalence and strongly projectively Morita equivalence. In particular, the notion of strongly projectively Morita equivalence which we introduce uses an interesting subclass of the w*-rigged modules, the strongly projectively w*-rigged modules which in turn define a subclass of the projectively w*-rigged modules, see \cite{Ble-Kr, Ble-Kas}. 
The notion of strongly projectively Morita equivalence also implies stable isomorphism between the corresponding w*-rigged modules in the sense of \cite{Ele-Paul-Tod} as well as between the dual operator algebras on which these modules act, whereas the notion of Morita equivalence does not necessarily imply stable isomorphism between either the corresponding modules or between the algebras.

Furthermore, we study a special class of right w*-rigged modules, the right w*-rigged modules over nest algebras and we prove relevant theorems for Morita equivalence between them.  The rest of this paper is organized as follows:

In Section 2 we recall some known definitions which are useful for the next sections. Also, we introduce the notion of a concrete strongly projectively w*-rigged module as well as the notion of a strongly projectively w*-rigged module. We provide a characterization for the class of concrete strongly projectively w*-rigged modules in the case the module is over a reflexive operator algebra, see Theorem \ref{proj-lat}.

Both new types of Morita equivalence between right w*-rigged modules over unital dual operator algebras are presented in Section 3. More specifically, if $\cl A$ and $\cl B$ are unital dual operator algebras and $E,\,F$ are right w*-rigged modules over $\cl A$ and $\cl B$ respectively, then we call $E$ and $F$ Morita equivalent if there exists a right w*-rigged module $Y$ over $\cl B$ such that \begin{enumerate}[(i)]
    \item $\cl A\cong Y\otimes_{\cl B}^{\sigma h}\tilde{Y}$ as dual operator $\cl A$-$\cl A$-bimodules.
    \item $\cl B\cong \tilde{Y}\otimes_{\cl A}^{\sigma h}Y$ as dual operator $\cl B$-$\cl B$-bimodules.
    \item $F\cong E\otimes_{\cl A}^{\sigma h}Y$ as right w*-rigged modules over $\cl B$ and $E\cong F\otimes_{\cl B}^{\sigma h}\tilde{Y}$ as right w*-rigged modules over $\cl A.$
\end{enumerate}
Here $\tilde{Y}$ is the counterpart bimodule of $\,Y$ and $\otimes_{\cl A}^{\sigma h}$ (resp. $\otimes_{\cl B}^{\sigma h}$) is the balanced normal Haagerup tensor product over $\cl A$ (resp. $\cl B$). For further details, see \cite{Ble-Kas-2}. 

From the conditions (i) and (ii) above we deduce that the dual operator algebras $\cl A$ and $\cl B$ are w*-Morita equivalent in the sense of \cite{Ble-Kas-2}.

We provide a counterexample to show that Morita equivalent right w*-rigged modules are not necessarily stably isomorphic.

The other type of Morita equivalence is strongly projectively Morita equivalence. If $E$ is a right w*-rigged module over a unital dual operator algebra $\cl A$ and $F$ is a right w*-rigged module over a unital dual operator algebra $\cl B$ then we call $E$ and $F$ strongly projectively Morita equivalent if there exists a strongly projectively right w*-rigged module $Y$ over $\cl B$ such that  \begin{enumerate}[(i)]
    \item $\cl A\cong Y\otimes_{\cl B}^{\sigma h}\tilde{Y}$ as dual operator $\cl A$-$\cl A$-bimodules.
    \item $\cl B\cong \tilde{Y}\otimes_{\cl A}^{\sigma h}Y$ as dual operator $\cl B$-$\cl B$-bimodules.
    \item $F\cong E\otimes_{\cl A}^{\sigma h}Y$ as right w*-rigged modules over $\cl B$ and $E\cong F\otimes_{\cl B}^{\sigma h}\tilde{Y}$ as right w*-rigged modules over $\cl A.$
\end{enumerate}

We will also prove that if $E$ and $F$ are strongly projectively Morita equivalent then $\cl A,\,\cl B$ are stably isomorphic (see Remark \ref{st}) and $E,\,F$ are also stably isomorphic (see Theorem \ref{stably isom}).

In Section 4, we focus on right w*-rigged modules over nest algebras. In \cite{Ele-app}, G.K. Eleftherakis characterized these modules. In particular he proved the following theorem.  

\begin{thm}\,\cite{Ele-app}.
\label{char}
Let $\,Y$ be a right dual operator module over a nest algebra $\Alg{\cl N}\subseteq \mathbb{B}(H).$ Then $\,Y$ is a right $\,\rm{w}^*$-rigged module over $\Alg{\cl N}$ if and only if there exist a Hilbert space $K,$ a $\,\rm{w}^*$-continuous and completely isometric right $\Alg{\cl N}$-module map $\Psi:Y\to \mathbb{B}(H,K),$ a nest $\cl M$ and a continuous onto nest homomorphism map $\phi:\cl N\to \cl M$ such that $\Psi(Y)$ is an $\Alg{\cl M}$-$\Alg{\cl N}$-bimodule and $$\Psi(Y)=Op(\phi):=\left\{y\in\mathbb{B}(H,K)\mid \phi(p)^{\perp}\,y\,p=0,\,\forall\,p\in\cl N\right\}$$ where, in general, for every projection $p$ of a Hilbert space $L$ we denote by $p^{\perp}$ the projection $Id_{L}-p.$
\end{thm}

We prove that if $\phi:\cl N_1\to \cl M_1$ and $\psi:\cl N_2\to M_2$ are continuous onto nest homomorphisms and $\chi:\cl N_1\to \cl N_2$ is a nest isomorphism such that $E=Op(\phi)\cong Op(\psi\circ \chi)$ as $\Alg{\cl N_1}$-right w*-rigged modules, then $E$ and $F=Op(\psi)$ are Morita equivalent.

Finally, we prove that if $\phi:\cl N_1\to \cl M_1$ and $\psi:\cl N_2\to \cl M_2$ are continuous onto nest homomorphisms and $\chi:\cl N_1\to \cl N_2$ is a nest isomorphism which can be extended to a $*$-isomorphism $\chi':\cl N_1^{\prime \prime}\to \cl N_2^{\prime \prime}$ such that $E=Op(\phi)\cong Op(\psi\circ \chi)$ as $\Alg{\cl N_1}$-right w*-rigged modules, then $E$ and $F=Op(\psi)$ are strongly projectively Morita equivalent.

\section{Preliminaries}
Our notation is standard. If $H$ and $K$ are Hilbert spaces then $\mathbb{B}(H,K)$ is the space of all linear and bounded operators from $H$ to $K$. If $H=K$ we write $\mathbb{B}(H,H)=\mathbb{B}(H).$ A dual operator algebra is an operator algebra which is also a dual operator space. Every w*-closed subalgebra of some $\mathbb{B}(H)$ is a dual operator algebra. Conversely, for any dual operator algebra $\cl A$ there exist a Hilbert space $H$ and a w*-continuous and completely isometric homomorphism $\alpha:\cl A\to \mathbb{B}(H).$ In this case we identify the algebra $\cl A$ with the w*-closed subalgebra $\alpha(\cl A)\subseteq \mathbb{B}(H).$ In the following, all dual operator algebras are unital, that is they possess an identity of norm $1.$ The diagonal of a dual operator algebra $\cl A$ is the $C^{*}$-algebra $\Delta(\cl A)=\cl A\cap \cl A^{*}.$ 

If $X$ is a subset of $\mathbb{B}(H,K)$ and $Y$ is a subset of $\mathbb{B}(K,L)$ then we denote by $\overline{[YX]}^{\text{w}^*}$ the w*-closure of the linear span of the set $$\left\{y\,x\in\mathbb{B}(H,L)\mid x\in X,\,y\in Y\right\}.$$

Furthermore, if $Z$ is a subset of $\mathbb{B}(L,R),$ then we denote by $\overline{[Z\,YX]}^{\text{w}^*}$ the w*-closure of the linear span of the set $$\left\{z\,y\,x\in\mathbb{B}(H,R)\mid x\in X,\,y\in Y,\,z\in Z\right\}.$$  A concrete right dual operator module $Y$ over a dual operator algebra $\cl A$ is a w*-closed subspace $Y\subseteq \mathbb{B}(H,K)$ such that $\overline{[Y\,\alpha(\cl A)]}^{\rm{w}^*}=Y$ for some w*-continuous and completely isometric unital homomorphism $\alpha:\cl A\to \mathbb{B}(H).$ 

An abstract right dual operator space $Y$ over a dual operator algebra $\cl A$ is defined to be a non-degenerate module over $\cl A$ which is also a dual operator space and the bilinear map $Y\times \cl A\to Y$ is separately w*-continuous with dense range. Similarly we define the two-sided bimodules which are considered to be non-degenerate both on the left and on the right. If $X$ is an operator space and $I,J$ are cardinals or sets, then we use the symbol $\mathbb{M}_{I,J}(X)$ for the operator space of $I\times J$ matrices over $X$ whose finite submatrices have uniformly bounded norm. We set $C_{J}^{w}(X)=\mathbb{M}_{J,1}(X)$ and $R_{J}^{w}(X)=\mathbb{M}_{1,J}(X)$ and these are written as $C_{n}(X)$ and $R_{n}(X)$ respectively if $J=n$ is finite. In the case where $X$ is a dual operator space, so is $\mathbb{M}_{I,J}(X)$ for any cardinals or sets $I,J$.  For further details we refer the reader to \cite{BleLeM04,Ble-Kas-2}. 

\begin{defn}\,\cite{Ele-Paul-Tod}.
\label{stable}
\,\,If $X,Y$ are dual operator spaces then we say that $X$ and $Y$ are \emph{stably isomorphic} if there exist a cardinal $J$ and a w*-continuous completely isometric map from $\mathbb{M}_{J}(X)$ onto $\mathbb{M}_{J}(Y)$. Here, $\mathbb{M}_{J}(X)=\mathbb{M}_{J,J}(X)$. Similarly $\mathbb{M}_{J}(Y)=\mathbb{M}_{J,J}(Y).$
\end{defn}

In \cite{Ele-Paul-Tod},  Eleftherakis,  Paulsen and Todorov proved that two dual operator spaces $X$ and $Y$ are stably isomorphic if and only if they are $\Delta$-equivalent. The special case of dual operator algebras has been studied in \cite{Ele-Paul,Ele-Paul-Tod}. We remind the relevant definitions.

\begin{defn}\,\cite{Ele-Paul-Tod}.
\label{TRO}
\,\,Let $X\subseteq \mathbb{B}(K_1,K_2)$ and $Y\subseteq \mathbb{B}(H_1,H_2)$ be dual operator spaces. We say that $X$ is \emph{TRO-equivalent} to $Y$ if there exist w*-closed TRO's $M_i\subseteq \mathbb{B}(H_i,K_i),\,i=1,2$ (i.e w*-closed linear subspaces of $\,\mathbb{B}(H_i,K_i)$ satisfying $M_i\,M_i^{*}M_i\subseteq M_i,\,i=1,2$) such that $$X=\overline{[M_2\,Y\,M_1^{*}]}^{\text{w}^*},\,Y=\overline{[M_2^{*}\,X\,M_1]}^{\text{w}^*}.$$
\end{defn}

\begin{defn}\,\cite{Ele-Paul-Tod}.
\label{delta}
Let $X$ and $Y$ be dual operator spaces. We say that $X$ \emph{is $\Delta$-equivalent} to $Y$ and we write $X\sim_{\Delta}Y$ if there exist w*-continuous and completely isometric maps $\Phi:X\to \mathbb{B}(K_1,K_2)$ and $\Psi:Y\to \mathbb{B}(H_1,H_2)$ for some Hilbert spaces $H_i,K_i,\,i=1,2,$ such that $\Phi(X)$ is TRO-equivalent to $\Psi(Y).$
\end{defn}

\begin{rem}
Other notions of Morita equivalence of dual operator algebras or dual operator spaces exist in \cite{Ble-Kas-1, Ele-mor-alg, Ele-ref,Ele-mor}.
\end{rem}

\subsection{w*-rigged modules}

\begin{defn}\,\cite{Ble-Kas}.
\label{bas-def}
\,Let $Y$ be a dual operator space which is a right operator module over a dual operator algebra $\cl A.$ Suppose that there exist a net of positive integers $(n(a))_{a\in J}$ and w*-continuous completely contractive right $\cl A$-module maps $$\phi_{a}:Y\to C_{n(a)}(\cl  A),\,\psi_{a}:C_{n(a)}(\cl A)\to Y$$ such that $\psi_{a}(\phi_{a}(y))\to y$ in the w*-topology on $Y,$ for all $y\in Y.$ Then we say that $Y$ is a \emph{right} w*-\emph{rigged module} over $\cl A.$  
\end{defn}
We also give the definition of a left w*-rigged module over a dual operator algebra.
\begin{defn}
\,Let $Y$ be a dual operator space which is a left operator module over a dual operator algebra $\cl A.$ Suppose that there exist a net of positive integers $(n(a))_{a\in J}$ and w*-continuous completely contractive left $\cl A$-module maps $$\phi_{a}:Y\to R_{n(a)}(\cl  A),\,\psi_{a}:R_{n(a)}(\cl A)\to Y$$ such that $\psi_{a}(\phi_{a}(y))\to y$ in the w*-topology on $Y,$ for all $y\in Y.$ Then we say that $Y$ is a \emph{left} w*-\emph{rigged module} over $\cl A$.  
\end{defn}

We identify two right (resp. left) w*-rigged modules $X,\,Y$ over a dual operator algebra $\cl A$ as $\cl A$-right (resp. left) dual operator modules, if there is a surjective w*-homeomorphic and completely isometric right (resp. left) $\cl A$-module map between them. In this case we write $X\cong Y$.

\begin{defn}
\label{bmaps}
\,\cite{Ble-Kas}.
If $\,X$ and $Y$ are right w*-rigged modules over a dual operator algebra $\cl A,$ then we denote by $B(X,Y)=\text{w}^{*} CB_{\cl A}(X,Y)$ the dual operator space of all w*-continuous and completely bounded right $\cl A$-module maps from $X$ to $Y$. If $X=Y$ then we write $B(Y)=\text{w}^{*} CB_{\cl A}(Y,Y)=\text{w}^{*} CB_{\cl A}(Y).$

\end{defn}

\begin{defn}\,\cite{Ble-Kas}.
For a right w*-rigged module $Y$ over a dual operator algebra $\cl A$ we denote by $\tilde{Y}$ the space $B(Y,\cl A)=\text{w}^{*}CB_{\cl A}(Y,\cl A).$

\end{defn}

According to \cite[Lemma 3.2]{Ble-Kas} we have that $\tilde{Y}$ is a w*-closed subspace of $CB_{\cl A}(Y,\cl A)$ and a left w*-rigged module over $\cl A$.

\begin{rem}\,\cite{Ble-Kas-2}.
\label{bimodule}
If $\cl A$ and $\cl B$ are dual operator algebras, $Y$ is an $\cl A$-$\cl B$-dual operator bimodule and $X$ is a $\cl B$-$\cl A$-dual operator bimodule, then $Y\otimes_{\cl B}^{\sigma h}X$ (resp. $X\otimes_{\cl A}^{\sigma h}Y$) is a dual operator $\cl A$-$\cl A$-(resp. $\cl B$) bimodule. In the case where $Y$ and $X$ are both right  w*-rigged modules over $\cl B$ and $\cl A$ repsectively, then $Y\otimes_{\cl B}^{\sigma h}X$ (resp. $X\otimes_{\cl A}^{\sigma h}Y$) is a right w*-rigged module over $\cl A$ (resp. $\cl B$), \cite[Subsection 3.3]{Ble-Kas}.
\end{rem}

We present now a subclass of right w*-rigged modules, the projectively right w*-rigged modules.

\begin{defn}\,\cite{Ble-Kas}.
\label{projectively}
\,If $\cl A\subseteq \mathbb{B}(H)$ is a dual operator algebra and $M\subseteq \mathbb{B}(H,K)$ is a w*-closed TRO such that $M^{*}\,M\subseteq \cl A$ then we call the space $Y=\overline{[M\,\cl A]}^{\text{w}^*}$ a \emph{projectively right} w*-\emph{rigged module} over $\cl A.$ 
\end{defn}

\begin{rem}
\label{tilde}
If $\cl A,\,M$ and $Y$ are as above, then by \cite[Example (8)]{Ble-Kas}, $\,Y$ is a right w*-rigged module in the sense of Definition \ref{bas-def} and $\tilde{Y}\cong \overline{[\cl A\,M^{*}]}^{\text{w}^*}.$ The converse is not true. Not every w*-rigged module is a projectively w*-rigged module. For further details we refer the reader to \cite[Proposition 4.4]{Ble-Kr} and \cite{Ele-mor-alg}. 
\end{rem}

We introduce the notion of a concrete strongly projectively right w*-rigged module as well as the notion of a strongly projectively right w*-rigged module over a unital dual operator algebra.

\begin{defn}
\label{strongly}
Let $\cl A\subseteq \mathbb{B}(H)$ be a unital dual operator algebra and let also $M\subseteq \mathbb{B}(H,K)$ be a w*-closed TRO such that $M^{*}\,M\subseteq \cl A$ and $1_{\overline{[M^{*}\,M]}^{\text{w}^*}}=1_{\cl A}.$ We call the dual operator space $Y=\overline{[M\,\cl A]}^{\text{w}^*}$ a \emph{concrete strongly projectively right} w*-\emph{rigged module} over $\cl A.$
\end{defn}

\begin{defn}
Let $\cl A$ be a unital dual operator algebra and let $Y$ be a dual operator space. We call $\,Y$ a \emph{strongly projectively right} w*-\emph{rigged module} over $\cl A$ if there exist a w*-continuous and completely isometric unital homomorphism $\alpha:\cl A\to \mathbb{B}(H)$ and a w*-continuous completely isometric right $\cl A$-module map $\Psi:Y\to \mathbb{B}(H,K)$ such that $\Psi(Y)$ is a concrete strongly projectively right w*-rigged module over $\alpha(\cl A).$
\end{defn}

\subsection{Nest algebras}
Let $\pr{\mathbb{B}(H)}$ be the set of all projections on $\mathbb{B}(H).$ A set $\cl L\subseteq \pr{\mathbb{B}(H)}$ which contains $0,I_{H}$ and arbitrary intersections and closed spans is called a lattice. We denote by $\Alg{\cl L}$ the unital dual operator algebra $$\Alg{\cl L}=\left\{x\in\mathbb{B}(H)\mid p^{\perp}\,x\,p=0,\,\forall\,p\in \cl L\right\}.$$
Dually, if $\cl A$ is a unital dual operator algebra acting on a Hilbert space $H$ the set $$\Lat{\cl A}=\left\{p\in \pr{\mathbb{B}(H)}\mid p^{\perp}\,a\,p=0,\,\forall\,a\in\cl A\right\}$$ is a lattice.
If $\cl A=\Alg{\Lat{\cl A}}$ then the algebra $\cl A$ is called reflexive. It is obvious that $\Lat{\cl A}\subseteq \Delta(\cl A)^{\prime},$ where $\Delta(\cl A)^{\prime}$ is the commutant of $\Delta(\cl A)=\cl A\cap \cl A^*.$

A nest $\cl N$ is a totally ordered set of projections of a Hilbert space $H,$ which contains $0,\,I_{H}$ and is closed under arbitrary intersections and closed spans. The corresponding nest algebra is $$\Alg{\cl N}=\left\{x\in\mathbb{B}(H) \mid p^{\perp}\,x\,p=0,\,\forall\,p\in\cl N\right\}.$$ 

\begin{rem}
\label{nest-lat}
A nest algebra $\cl A=\Alg{\cl N}$ is a reflexive algebra and $\cl N=\Lat{\cl A},$ see \cite{Dav}. Furthermore, by \cite[Corollary 22.18]{Dav} we have that $\cl N^{\prime}=\Delta(\cl A)$ and as a consequence, $\cl N^{\prime \prime}=\Delta(\cl A)^{\prime}.$
\end{rem}

Let $\cl N,\,\cl M$ be nests acting on the Hilbert spaces $H,\,K$ respectively. An order preserving map $\phi:\cl N\to \cl M$ is called a nest homomorphism. If this map is also surjective and injective, it is called a nest isomorphism. In the case where $\phi:\cl N\to \cl M$ is a continuous (with respect to $\mathbf{WOT}$ topologies on $\pr{\mathbb{B}(H)}$ and $\pr{\mathbb{B}(K)}$ respectively) onto nest homomorphism the space $$Op(\phi)=\left\{y\in\mathbb{B}(H,K)\mid \phi(p)^{\perp}\,y\,p=0,\,\forall\,p\in\cl N\right\}$$ is an $\Alg{\cl N}$-right w*-rigged module which is also a left module over $\Alg{\cl M},$ see \cite[Theorem 2.1]{Ele-app}.

\begin{rem}\,\cite[Theorem 2.9]{Ele-mor}.
\label{for-ref}
\,Let $\,\cl N$ and $\cl M$ be nests acting on Hilbert spaces $H$ and $K$ respectively. The dual operator algebras $\cl A=\Alg{\cl N}$ and $\cl B=\Alg{\cl  M}$ are w*-Morita equivalent in the sense of  \cite{Ble-Kas-2}  if, and only if, there exists a nest isomorphism $\theta:\cl N\to \cl M$. 

In this case, if $\,Y=Op(\theta^{-1})$ and $X=Op(\theta),$ then \begin{enumerate}[(1)]
    \item $\cl A\cong Y\otimes_{\cl B}^{\sigma h}X$ as dual operator $\cl A$-$\cl A$-bimodules.
    \item $\cl B\cong X\otimes_{\cl A}^{\sigma h}Y$ as dual operator $\cl B$-$\cl B$-bimodules.
    \item $\tilde{Y}\cong X$ as right dual operator modules over $\cl B,$ \cite[Definition 4.2]{Ble-Kas}.
\end{enumerate}
\end{rem}

\begin{rem}
\label{map}

Let $\,Y$ be a subspace of $\,\mathbb{B}(H,K)$. We denote by $\Map{Y}$ the map which sends every $p\in\pr{\mathbb{B}(H)}$ to the projection of $K$ generated by the vectors of the form $y\,p(\xi)$ where $y\in Y$ and $\xi\in H$.

\end{rem}

The following theorem provides a characterization for the class of concrete strongly projectively right w*-rigged modules over reflexive operator algebras.

\begin{thm} 
\label{proj-lat}
Let $\,Y$ be a dual operator space and let $\,\cl A\subseteq \mathbb{B}(H)$ be a reflexive operator algebra. The following are equivalent.
\begin{enumerate}[(i)]
    \item $\,Y$ is a concrete strongly projectively right $\mathrm{w^*}$-rigged module over $\cl A.$
    \item There exist a reflexive operator algebra $\cl B\subseteq \mathbb{B}(K)$ and a $*$-isomorphism $\theta:\Delta(\cl A)^{\prime}\to \Delta(\cl B)^{\prime}$ such that $\theta(\Lat{\cl A})=\Lat{\cl B}$ and $$Y=\left\{y\in\mathbb{B}(H,K)\mid  \theta(p)^{\perp}\,y\,p=0,\,\forall\,p\in\Lat{\cl A}\right\}.$$
\end{enumerate}

\end{thm}

\begin{proof}

$(i)\implies (ii)$
Let $M\subseteq \mathbb{B}(H,K)$ be a w*-closed TRO such that $M^{*}\,M\subseteq \cl A,\,1_{\overline{[M^{*}\,M]}^{\text{w}^*}}=1_{\cl A}$ and $Y=\overline{[M\,\cl A]}^{\text{w}^{*}}.$ We consider the dual operator algebra $\cl B=\overline{[M\,\cl A\,M^{*}]}^{\text{w}^{*}}\subseteq \mathbb{B}(K).$ Since $M^*\,M\subseteq \cl A$ and $1_{\overline{[M^{*}\,M]}^{\text{w}^*}}=1_{\cl A}$ we have that $$\overline{[M^{*}\,\cl B\,M]}^{\text{w}^{*}}=\overline{[M^*\,M\,\cl A\,M^*\,M]}^{\text{w}^*}=\cl A.$$
Therefore, the algebras $\cl A$ and $\cl B$ are TRO-equivalent and by \cite[Remark 2.7]{Ele-TRO} it follows that $\cl B$ is also a reflexive algebra. According to
\cite[Proposition 2.5]{Ele-TRO} we also have that 
$$\Delta(\cl A)=\overline{[M^{*}\,\Delta(\cl B)\,M]}^{\text{w}^{*}},\,\Delta(\cl B)=\overline{[M\,\Delta(\cl A)\,M^{*}]}^{\text{w}^{*}}.$$ From \cite[Lemma 2.6]{Ele-TRO}, the map $\chi=\Map{M}:\pr{\Delta(\cl A)'}\to \pr{\Delta(\cl B)'}$ can be extended to a $*$-isomorphism $\theta:\Delta(\cl A)'\to \Delta(\cl B)'$ such that $\theta(\Lat{\cl A})=\Lat{\cl B}.$ If we define the w*-closed TRO $$Z=\left\{z\in\mathbb{B}(H,K)\mid z\,p=\theta(p)\,z,\,\forall\,p\in\pr{\Delta(\cl A)'}\right\}$$ then by \cite[Proposition 2.8]{Ele-TRO} we have that $M\subseteq Z$ and $$\cl A=\overline{[Z^{*}\,\cl B\,Z]}^{\text{w}^{*}},\,\cl B=\overline{[Z\,\cl A\,Z^{*}]}^{\text{w}^{*}}$$ $$\Delta(\cl A)=\overline{[Z^{*}\,Z]}^{\text{w}^{*}},\,\Delta(\cl B)=\overline{[Z\,Z^{*}]}^{\text{w}^{*}}.$$ We aim to show that $$Y=\left\{y\in\mathbb{B}(H,K)\mid  \theta(p)^{\perp}\,y\,p=0,\,\forall\,p\in\Lat{\cl A}\right\}.$$ We set $Y_0=\left\{y\in\mathbb{B}(H,K)\mid  \theta(p)^{\perp}\,y\,p=0,\,\forall\,p\in\Lat{\cl A}\right\}$ and we need to show that $Y=Y_0.$ We observe that $Y_0=\overline{[Z\,\cl A]}^{\text{w}^{*}}.$ Indeed, if $z\,a\in Z\,\cl A$ then for all $p\in\Lat{\cl A}$ we have that $$\theta(p)^{\perp}\,z\,a\,p=z\,a\,p-\theta(p)\,z\,a\,p=z\,a\,p-z\,p\,a\,p=z\,p^{\perp}\,a\,p=0$$ so $Z\,\cl A\subseteq Y_0$ and thus $\overline{[Z\,\cl A]}^{\text{w}^{*}}\subseteq Y_0.$ On the other hand, since $\cl B$ is unital we get that $Y_0\subseteq \overline{[\cl B\,Y_0]}^{\text{w}^{*}}=\overline{[Z\,\cl A\,Z^{*}\,Y_0]}^{\text{w}^{*}},$  so it suffices to prove that $Z^{*}\,Y_0\subseteq \cl A$. If $p\in\Lat{\cl A},\,z\in Z$ and $y\in Y_0$ we have that
    $$p^{\perp}\,z^{*}\,y\,p=z^{*}\,y\,p-p\,z^{*}\,y\,p=z^{*}\,y\,p-z^{*}\,\theta(p)\,y\,p=z^{*}\,\theta(p)^{\perp}\,y\,p=0$$
which means that $z^{*}\,y\in \Alg{\Lat{\cl A}}=\cl A$. Therefore $Z^{*}\,Y_0\subseteq \cl A$ as desired. Now, by the fact that $M\subseteq Z$ we have that $M\,\cl A\subseteq Z\,\cl A$ which implies that \begin{equation} \label{sub1}Y\subseteq Y_0.\end{equation} Moreover, \begin{equation} \label{sub2} Y_0\subseteq \overline{[Y_0\,M^{*}\,M]}^{\text{w}^{*}}\subseteq \overline{[\cl B\,Z\,M^{*}\,M]}^{\text{w}^{*}}\subseteq \overline{[\cl B\,M]}^{\text{w}^*}=Y\end{equation} since 
    $Z\,M^{*}\subseteq Z\,Z^{*}\subseteq \Delta(\cl B)\subseteq \cl B$.
    
By (\ref{sub1}) and (\ref{sub2}) we have that $Y=Y_0,$ as desired.

$(ii)\implies (i)$ Let $\cl B\subseteq \mathbb{B}(K)$ be a reflexive operator algebra and let $\,\theta:\Delta(\cl A)^{\prime}\to \Delta(\cl B)^{\prime}$ be a $*$-isomorphism such that $\theta(\Lat{\cl A})=\Lat{\cl B}$ and $$Y=\left\{y\in\mathbb{B}(H,K)\mid  \theta(p)^{\perp}\,y\,p=0,\,\forall\,p\in\Lat{\cl A}\right\}.$$ Since $\Delta(\cl A)^{\prime}$ and $\,\Delta(\cl B)^{\prime}$ are von Neumann algebras, from  \cite[Theorem 3.2]{Ele-TRO} we have that the space $$M=\left\{m\in\mathbb{B}(H,K) \mid m\,x=\theta(x)\,m,\,\forall\,x\in \Delta(\cl A)^{\prime}\right\}$$ is a w*-closed TRO such that $$\cl B=\overline{[M\,\cl A\,M^{*}]}^{\text{w}^{*}},\,\cl A=\overline{[M^{*}\,\cl B\,M]}^{\text{w}^{*}}.$$ It follows that $M^{*}\,M\subseteq \cl A$ and $1_{\cl A}=1_{\overline{[M^*\,M]}^{\text{w}^{*}}}.$ Indeed, since the unit of $\overline{[M^*\,M]}^{\text{w}^*}$ acts as a unit on the left and on the right on $\cl A=\overline{[M^*\,\cl B\,M]}^{\text{w}^*}$ (as $M$ is a TRO) and since it is a subalgebra of $\cl A,$ then the unit of $\overline{[M^*\,M]}^{\text{w}^*}$ is also a unit of $\cl A.$ We aim to show that $Y=\overline{[M\,\cl A]}^{\text{w}^{*}}.$ Indeed, for all $p\in\Lat{\cl A},\,m\in M$ and $a\in\cl A,$ it holds that $$\theta(p)^{\perp}\,m\,a\,p=m\,a\,p-\theta(p)\,m\,a\,p=m\,a\,p-m\,p\,a\,p=m\,p^{\perp}\,a\,p=0$$ and thus \begin{equation}\label{sub3}\overline{[M\,\cl A]}^{\text{w}^{*}}\subseteq Y.\end{equation}  Since $\cl B$ is unital we get that $Y\subseteq \overline{[\cl B\,Y]}^{\text{w}^{*}}=\overline{[M\,\cl A\,M^{*}\,Y]}^{\text{w}^{*}}.$ If $p\in\Lat{\cl A}$ and $m\in M,\,y\in Y$ then $$p^{\perp}\,m^{*}\,y\,p=m^{*}\,\theta(p)^{\perp}\,y\,p=0$$ which means that $M^{*}\,Y\subseteq \Alg{\Lat{\cl A}}=\cl A.$ Therefore \begin{equation}
    \label{sub4}Y\subseteq \overline{[M\,\cl A\,M^*\,Y]}^{\text{w}^{*}}\subseteq \overline{[M\,\cl A\,\cl A]}^{\text{w}^*}=\overline{[M\,\cl A]}^{\text{w}^*}.\end{equation}
By (\ref{sub3}) and (\ref{sub4}) we deduce that $Y=\overline{[M\,\cl A]}^{\text{w}^{*}}$ and as a consequence, $Y$ is a concrete strongly projectively right w*-rigged module over $\cl A.$
\end{proof}

By combining Theorem \ref{proj-lat} and Remark \ref{nest-lat} we take the following.
\begin{cor}
\label{proj-nest}
Let $\cl N$ be a nest acting on the Hilbert space $H$ with corresponding nest algebra $\,\cl A=\Alg{\cl N}\subseteq \mathbb{B}(H)$ and let $\,Y$ be a dual operator space. The following are equivalent.
\begin{enumerate}[(i)]
    \item $\,Y$ is a concrete strongly projectively right $\mathrm{w^*}$-rigged module over $\cl A$.
    \item There exist a nest $\cl M$ acting on some Hilbert space $K$ and a $*$-isomorphism $\theta:\cl N^{\prime \prime}\to \cl M^{\prime \prime}$ such that $\theta(\cl N)=\cl M$ and $$Y=\left\{y\in\mathbb{B}(H,K)\mid \theta(p)^{\perp}\,y\,p=0,\,\forall\,p\in\cl N\right\}.$$ 
\end{enumerate}
\end{cor}

\section{Morita equivalence}
Throughout this section we develop a theory of Morita equivalence for right w*-rigged modules over unital dual operator algebras.

Let $\cl B$ be a unital dual operator algebra and let $Y$ be a right w*-rigged module over $\cl B.$ By \cite[Lemma 3.2]{Ble-Kas} we get that $\tilde{Y}$ is a left w*-rigged module over $\cl B$ and the map $(\cdot\,,\cdot):\tilde{Y}\times Y\to \cl B$ which sends every $(f,y)\in\tilde{Y}\times Y$ to $f(y)\in\cl B$ is separately w*-continuous and completely contractive. From the discussion above \cite[Theorem 3.5]{Ble-Kas}, the space $\cl A=Y\otimes_{\cl B}^{\sigma h}\tilde{Y}$ is a dual operator algebra with identity of norm 1 and by \cite[Definition 4.3]{Ble-Kas}, $\,Y$ is a non-degenerate left module over $\cl A$ as well as $\tilde{Y}$ is a non-degenerate right module over $\cl A.$  

We recall from Definition \ref{bmaps} that for a right w*-rigged module $Y$ over a unital dual operator algebra $\cl B$ we denote by $B(Y)$ the dual operator space of all w*-continuous and completely bounded right $\cl B$-module maps from $Y$ to $Y.$ In \cite[Theorem 3.5]{Ble-Kas} it is proven that $\cl A\cong B(Y)$ as dual operator algebras and thus the induced non-degenerate left action of $\cl A$ on $Y$ is the following $$\cl A\times Y\to Y,\,(a,y)\mapsto a\,y=\theta(a)(y),$$ where $\theta:Y\otimes_{\cl B}^{\sigma h}\tilde{Y}\to B(Y)$ is the isomorphism proved in the above theorem such that $$\theta(y\otimes_{\cl B}f)(y')=y\,f(y'),\,y\in Y,\,f\in\tilde{Y}.$$
Furthermore, $B(Y)$ is an $\cl A$-$\cl A$-bimodule with actions $$a\,T:Y\to Y,\,(a\,T)(y)=a\,T(y),\,\forall\,a\in\cl A,\,\forall\,T\in B(Y)$$ $$T\,a:Y\to Y,\,(T\,a)(y)=T(a\,y),\,\forall\,a\in\cl A,\,\forall\,T\in B(Y).$$
We claim that $\theta$ is also a bimodule isomorphism in a canonical way. Since $\theta$ is w*-continuous it suffices to prove that $$\theta(a\,(y\otimes_{\cl B}f))=a\,\theta(y\otimes_{\cl B}f)\,\,\,\text{and}\,\,\,\theta((y\otimes_{\cl B}f)\,a)=\theta(y\otimes_{\cl B}f)\,a,$$ where $a\in\cl A,\,y\in Y,\,f\in\tilde{Y}.$ Indeed, for every $y'\in Y$ we have that \begin{align*}
    \theta(a\,(y\otimes_{\cl B}f))(y')&=\theta(a\,y\otimes_{\cl B}f)(y')\\&=a\,y\,f(y')\\&=a\,(y\,f(y'))\\&=(a\,\theta(y\otimes_{\cl B}f))(y')
\end{align*}
and \begin{align*}
    \theta((y\otimes_{\cl B}f)\,a)(y')&=\theta(y\otimes_{\cl B}f\,a)(y')\\&=y\,(f\,a)(y')\\&=y\,f(a\,y')\\&=\theta(y\otimes_{\cl B}f)(a\,y')\\&=(\theta(y\otimes_{\cl B}f)\,a)(y').
\end{align*}

\begin{rem}
\label{explain}
Let $\cl A$ and $\cl B$ be unital dual operator algebras and $Y$ be a right w*-rigged module over $\cl B$ such that $\cl A\cong Y\otimes_{\cl B}^{\sigma h}\tilde{Y}$ as dual operator $\cl A$-bimodules (that is, via a completely isometric w*-homeomorphism which is also an $\cl A$-bimodule map) and also
$\cl B\cong \tilde{Y}\otimes_{\cl A}^{\sigma h}Y$ as dual operator $\cl B$-bimodules. 

Therefore, the $\cl B$-$\cl A$-bimodule $\tilde{Y}$ is a w*-Morita equivalence bimodule in the sense of \cite{Ble-Kas-2} and by the last paragraph above \cite[Definition 1.3]{Ble-Kas-rigged} we deduce that $\tilde{Y}$ is a right w*-rigged module over $\cl A$. 
\end{rem}

Now we can give the definition of Morita equivalence for right w*-rigged modules.
\begin{defn}
\label{mor-eq}
Let $\cl A$ and $\cl B$ be dual operator algebras, $E$ be a right w*-rigged module over $\cl A$ and $F$ be a right w*-rigged module over $\cl B$. We call $E,\,F$ \emph{Morita equivalent} if there exists a right w*-rigged module $Y$ over $\cl B$ such that 
\begin{enumerate}[(i)]
    \item $\cl A\cong Y\otimes_{\cl B}^{\sigma h}\tilde{Y}$ as dual operator $\cl A$-$\cl A$-bimodules.
    \item $\cl B\cong \tilde{Y}\otimes_{\cl A}^{\sigma h}Y$ as dual operator $\cl B$-$\cl B$-bimodules.
    \item $F\cong E\otimes_{\cl A}^{\sigma h}Y$ as right w*-rigged modules over $\cl B$ and $E\cong F\otimes_{\cl B}^{\sigma h}\tilde{Y}$ as right w*-rigged modules over $\cl A.$
\end{enumerate}
\end{defn}

\begin{rem}
If $E,\,F,\,\cl A,\,\cl B$ are as in the Definition above, then by (i), (ii), Remark \ref{explain} and the discussion above it we get that $\cl A\cong B(Y),\,\cl B\cong B(\tilde{Y})$ and also $\cl A,\,\cl B$ are w*-Morita equivalent in the sense of \cite{Ble-Kas-2} with equivalence bimodules $Y$ and $\tilde{Y}.$
\end{rem}

\begin{rem}
\label{not-st}
Let $E,\,F,\,\cl A,\,\cl B$ and $\,Y$ be as in the Definition \ref{mor-eq}. We observe that in the condition (iii) it suffices to prove one of the two isomorphisms since either of them implies the other one. Indeed, suppose that $F\cong E\otimes_{\cl A}^{\sigma h}Y$ as right w*-rigged modules over $\cl B.$ Using the condition (i), the associativity of the normal Haagerup tensor product and the fact that $E$ is a non-degenerate $\cl A$-right module we have that $$F\otimes_{\cl B}^{\sigma h}\tilde{Y}\cong \left(E\otimes_{\cl A}^{\sigma h}Y\right)\otimes_{\cl B}^{\sigma h}\tilde{Y}\cong E\otimes_{\cl A}^{\sigma h}\left(Y\otimes_{\cl B}^{\sigma h}\tilde{Y}\right)\cong E\otimes_{\cl A}^{\sigma h}\cl A\cong E.$$ 
Similarly, if $E\cong F\otimes_{\cl B}^{\sigma h}\tilde{Y}$ as right w*-rigged modules over $\cl A$ then by (ii), the associativity of the normal Haagerup tensor product and the fact that $F$ is a non-degenerate $\cl B$-right module we get that $F\cong E\otimes_{\cl A}^{\sigma h}Y$ as right w*-rigged modules over $\cl B.$ 

However, $E$ and $F$ are not always stably isomorphic. Indeed, by \cite[Example 3.7]{Ele-ref} there exist isomorphic nests $\cl N$ and $\cl M$ such that the nest algebras $\cl A=\Alg{\cl N}$ and $\cl B=\Alg{\cl M}$ are not stably isomorphic (since they are not $\Delta$-equivalent). Let $\theta:\cl N\to \cl M$ be a nest isomorphism. If $Y=Op(\theta)$ then by Remark \ref{for-ref} we have that $\tilde{Y}\cong X=Op(\theta^{-1})$ and also $\cl A\cong Y\otimes_{\cl B}^{\sigma h}\tilde{Y},\,\cl B\cong \tilde{Y}\otimes_{\cl A}^{\sigma h}Y.$ By defining $E=\cl A$ over $\cl A$ and $F=Y$ over $\cl B$ we have that $E$ and $F$ are Morita equivalent in our sense but not stably isomorphic.
\end{rem}

\begin{defn}
\label{d-mor}
Let $\cl A$ and $\cl B$ be dual operator algebras, $E$ be a right w*-rigged module over $\cl A$ and $F$ be a right w*-rigged module over $\cl B$. We call $E,\,F$ \emph{strongly projectively Morita equivalent} if there exists a strongly projectively right w*-rigged module $Y$ over $\cl B$ such that \begin{enumerate}[(i)]
    \item $\cl A\cong Y\otimes_{\cl B}^{\sigma h}\tilde{Y}$ as dual operator $\cl A$-$\cl A$-bimodules.
    \item $\cl B\cong \tilde{Y}\otimes_{\cl A}^{\sigma h}Y$ as dual operator $\cl B$-$\cl B$-bimodules.
    \item $F\cong E\otimes_{\cl A}^{\sigma h}Y$ as right w*-rigged modules over $\cl B$ and $E\cong F\otimes_{\cl B}^{\sigma h}\tilde{Y}$ as right w*-rigged modules over $\cl A.$
\end{enumerate}
\end{defn}

\begin{rem}
\label{st}
If $\cl A,\,\cl B,\,E,\,F$ and $Y$ are as in the Definition \ref{d-mor}, then since $Y$ is a strongly projectively right w*-rigged module over $\cl B,$ there exist a w*-continuous completely isometric unital homomorphism $\beta:\cl B\to \mathbb{B}(H)$ and a w*-closed TRO $M\subseteq \mathbb{B}(H,K)$ such that $M^{*}\,M\subseteq \beta(\cl B),\,1_{\overline{[M^{*}\,M]}^{\text{w}^{*}}}=1_{\beta(\cl B)}$ and $Y\cong \overline{[M\,\beta(\cl B)]}^{\text{w}^{*}}$ as right w*-rigged modules over $\cl B.$ Define the unital dual operator algebra $\cl C=\overline{[M\,\beta(\cl B)\,M^{*}]}^{\text{w}^{*}}.$ Since $M^*\,M\subseteq \beta(\cl B)$ and $\overline{[M^*\,M]}^{\text{w}^*},\,\beta(\cl B)$ share the same unit we have that $$\overline{[M^*\,M\,\beta(\cl B)]}^{\text{w}^*}=\beta(\cl B)=\overline{[\beta(\cl B)\,M^*\,M]}^{\text{w}^*}$$ which implies that $$\overline{[M^*\,\cl C\,M]}^{\text{w}^*}=\overline{[M^*\,M\,\beta(\cl B)\,M^*\,M]}^{\text{w}^*}=\beta(\cl B).$$ Therefore, the algebras $\beta(\cl B)$ and $\cl C$ are TRO-equivalent via $M.$ By \cite[Theorem 2.4]{Ele-Paul}, for the $M$-generated $\cl B$-$\cl C$ bimodules $$U=\overline{[\beta(\cl B)\,M^*]}^{\text{w}^*}\cong \tilde{Y},\,\,V=\overline{[M\,\beta(\cl B)]}^{\text{w}^*}\cong Y$$ it holds that $$\beta(\cl B)\cong U\otimes_{\cl C}^{\sigma h}V,\,\,\cl C\cong V\otimes_{\beta(\cl B)}^{\sigma h}U\cong V\otimes_{\cl B}^{\sigma h}U$$
Therefore, $\cl C\cong V\otimes_{\cl B}^{\sigma h}U\cong Y\otimes_{\cl B}^{\sigma h}\tilde{Y}\cong \cl A.$ Since $\beta(\cl B)$ and $\cl C$ are TRO-equivalent and $\cl C\cong \cl A,$ we deduce that $\cl A$ and $\cl B$ are $\Delta$-equivalent, i.e $\cl A$ and $\cl B$ are stably isomorphic, \cite[Theorem 3.2]{Ele-Paul}.
\end{rem}

\begin{thm}
\label{stably isom}
Let $E$ and $F$ be right $\mathrm{w^{*}}$-rigged modules over the unital dual operator algebras $\cl A$ and $\cl B$ respectively. If $\,E$ and $F$ are strongly projectively Morita equivalent then $E$ and $F$ are stably isomorphic.
\end{thm}

\begin{proof}
Let $E$ and $F$ be strongly projectively Morita equivalent. There exists a strongly projectively right w*-rigged module $Y$ over $\cl B$ such that $$\cl A\cong Y\otimes_{\cl B}^{\sigma h}\tilde{Y},\,\cl B\cong \tilde{Y}\otimes_{\cl A}^{\sigma h}Y,\,F\cong E\otimes_{\cl A}^{\sigma h}Y.$$ Since $Y$ is a strongly projectively right w*-rigged module over $\cl B,$ there exist a w*-continuous completely isometric unital homomorphism $\beta:\cl B\to \mathbb{B}(H)$ and a w*-closed TRO $M\subseteq \mathbb{B}(H,K)$ such that $M^{*}\,M\subseteq \beta(\cl B),\,1_{\overline{[M^*\,M]}^{\text{w}^*}}=1_{\beta(\cl B)}$ and $Y\cong \overline{[M\,\beta(\cl B)]}^{\text{w}^{*}}.$ By Remark \ref{st} it follows that the algebras $\cl C=\overline{[M\,\beta(\cl B)\,M^*]}^{\text{w}^*}$ and $\beta(\cl B)$ are TRO-equivalent via $M.$ By \cite[Lemma 3.1]{Ele-Paul},  for the $\beta(\cl B)$-$\cl C$-bimodule $\tilde{Y}\cong \overline{[\beta(\cl B)\,M^{*}]}^{\text{w}^*}$ we have that $$R_{I}^{w}(\tilde{Y})\cong R_{I}^{w}(\beta(\cl B))\cong R_{I}^{w}(\cl B),$$ where $I$ is an infinite indexed set. Since $F\cong E\otimes_{\cl A}^{\sigma h}Y,$  by Remark \ref{not-st} we get that $E\cong F\otimes_{\cl B}^{\sigma h}\tilde{Y}$ as $\cl A$-right w*-rigged modules and from \cite[Lemma 2.7]{Ble-Kas-2} we deduce that \begin{align*}
   \mathbb{M}_{I}(E)&\cong \mathbb{M}_{I}\left(F\otimes_{\cl B}^{\sigma h}\tilde{Y}\right)\\&\cong C_{I}^{w}(F)\otimes_{\cl B}^{\sigma h}R_{I}^{w}(\tilde{Y})\\&\cong C_{I}^{w}(F)\otimes_{\cl B}^{\sigma h}R_{I}^{w}(\cl B)\\&\cong \mathbb{M}_{I}(F\otimes_{\cl B}^{\sigma h}\cl B)\\&\cong \mathbb{M}_{I}(F).
\end{align*}
Therefore, $E$ and $F$ are stably isomorphic.
\end{proof}

\begin{ex}
We provide an example of strongly projectively Morita equivalent w*-rigged modules. Let $M\subseteq \mathbb{B}(H,K)$ be a w*-closed TRO, where $H$ and $K$ are Hilbert spaces and let $E$ be a right w*-rigged module over the unital dual operator algebra $\cl A=\overline{[M\,M^{*}]}^{\text{w}^*}\subseteq \mathbb{B}(K).$ By defining $F=E\otimes_{\cl A}^{\sigma h}M$ we have that $F$ is a right w*-rigged module over $\cl B=\overline{[M^*\,M]}^{\text{w}^*}\subseteq \mathbb{B}(H)$ (see Remark \ref{bimodule}) and $E,\,F$ are strongly projectively Morita equivalent with equivalence bimodules $Y=M=\overline{[M\,\cl B]}^{\text{w}^*}$ and $\tilde{Y}\cong M^*=\overline{[\cl B\,M^*]}^{\text{w}^*}.$
\end{ex}

\section{Morita equivalence of w*-rigged modules over nest algebras}

As mentioned in the introduction, in this section we study the Morita equivalence and the strongly projectively Morita equivalence between right w*-rigged modules over nest algebras. The main theorems of this section are Theorem \ref{morita} and Theorem \ref{d-morita}. Before we state them, we prove some useful lemmas. For the first lemma we recall the Remark \ref{map}, where for a subspace $Y\subseteq \mathbb{B}(H,K)$ we denote by $\Map{Y}$ the map which sends every projection $p\in\pr{\mathbb{B}(H)}$ to the projection of $K$ generated by the vectors of the form $y\,p\,(\xi)$ where $y\in Y$ and $\xi\in H$.

\begin{lem}
\label{help}
Let $\cl M$ be a nest acting on some Hilbert space $K$ and let $\Alg{\cl M}$ be the corresponding nest algebra. If $p\in\cl M$ and $q$ is the projection onto $\overline{[\Alg{\cl M}\,p(K)]}^{||\cdot||}=\overline{[x\,p(\xi)\in K\mid x\in\Alg{\cl M},\,\xi\in K]}^{||\cdot||}$ then $p=q.$
\end{lem}

\begin{proof} 
    Since $I_{K}\in\Alg{\cl M}$ it follows that $p\leq q.$ On the other hand it holds that $p^{\perp}\,\Alg{\cl M}\,p=0$ which implies that $$\Alg{\cl M}\,p(K)=p\,\Alg{\cl M}\,p(K)\subseteq p(K)$$ that is $q\leq p.$ We deduce that $p=q.$
\end{proof}

\begin{lem}
\label{restriction}
Let $\cl N,\,\cl M$ be nests acting on the Hilbert spaces $H,\,K$ respectively. If $\phi:\cl N\to \cl M$ is a continuous onto nest homomorphism and $$U=Op(\phi)=\left\{u\in\mathbb{B}(H,K)\mid \phi(p)^{\perp}\,u\,p=0,\,\forall\,p\in\cl N\right\}$$ then $\Map{U}|_{\cl N}=\phi.$
\end{lem}

\begin{proof}
Set $\phi'=\Map{U}.$ We will show that $\phi'(p)=\phi(p),\,\forall\,p\in\cl N.$ Denote by $V$ the space $$V=\left\{v\in\mathbb{B}(K,H)\mid p^{\perp}\,v\,\phi(p)=0,\,\forall\,p\in\cl N\right\}$$ and let $\psi=\Map{V}.$ By \cite[Theorem 2.1]{Ele-app} we have that $\Alg{\cl M}=\overline{[U\,V]}^{\text{w}^*}.$ Let $p\in\cl N.$ Since $\phi(p)^{\perp}\,U\,p=0$ it follows that $\phi(p)^{\perp}\,\phi'(p)=0$ and thus \begin{equation}
    \label{4.1}\phi'(p)\leq \phi(p).
\end{equation} On the other hand, $p^{\perp}\,V\,\phi(p)=0$ which implies that $V\,\phi(p)(K)\subseteq p$ and as a consequence \begin{equation}
    \label{4.2}\psi(\phi(p))\leq p.
\end{equation} For all $u\in U,\,v\in V$ we have that \begin{align*}
    u\,v\,\phi(p)&=u\,\psi(\phi(p))\,v\,\phi(p)\\&=\phi'(\psi(\phi(p)))\,u\,\psi(\phi(p))\,v\,\phi(p)\\&=\phi'(\psi(\phi(p)))u\,v\,\phi(p)
\end{align*} which means that \begin{equation}
    \label{4.3}\overline{[U\,V\,\phi(p)(K)]}^{||\cdot||}\subseteq \phi'(\psi(\phi(p)))(K).
\end{equation} 
Since $\Alg{\cl M}=\overline{[U\,V]}^{\text{w}^*}$ and $\phi(p)\in\cl M,$ from Lemma \ref{help}, $\phi(p)$ is the projection onto $\overline{[U\,V\,\phi(p)(K)]}^{||\cdot||}$ and  
thus by \ref{4.3} we get that $\phi(p)\leq \phi'(\psi(\phi(p)))$. Using \ref{4.2} we have that \begin{equation}
    \label{4.4}\phi(p)\leq \phi'(p).
\end{equation}
By combining \ref{4.1} and \ref{4.4} we deduce that $\phi(p)=\phi'(p).$
\end{proof}

\begin{lem}
\label{composition}

Let $\cl N,\,\cl M,\,\cl L$ be nests acting on the Hilbert spaces $H,\,K$ and $R$ respectively. Let also $\phi:\cl N\to \cl M,\,\zeta:\cl N\to \cl L$ and $\theta:\cl M\to \cl L$ be continuous onto nest homomorphisms. The following are equivalent.
\begin{enumerate}[(i)]
    \item $\zeta=\theta\circ \phi.$
    \item $Op(\zeta)=\overline{[Op(\theta)\,Op(\phi)]}^{\mathrm{w}^*}.$
\end{enumerate}

\end{lem}

\begin{proof}
$(i)\implies (ii)$ Let $\zeta=\theta\circ \phi.$ For all $x\in Op(\theta),\,y\in Op(\phi)$ we have that \begin{align*}
    \zeta(p)^{\perp}\,x\,y\,p&=\theta(\phi(p))^{\perp}\,x\,y\,p\\&=\theta(\phi(p))^{\perp}\,x\,(\phi(p)+\phi(p)^{\perp})\,y\,p\\&=\theta(\phi(p))^{\perp}\,x\,\phi(p)\,y\,p+\theta(\phi(p))^{\perp}\,x\,\phi(p)^{\perp}\,y\,p\\&=0,\,\forall\,p\in\cl N
\end{align*}
since $\theta(\phi(p))^{\perp}\,x\,\phi(p)=0$ ($\phi(p)\in \cl M,\,x\in Op(\theta)$) and $\phi(p)^{\perp}\,y\,p=0$ ($y\in Op(\phi)$)
which means that $x\,y\in Op(\zeta).$ Thus, $Op(\theta)\,Op(\phi)\subseteq Op(\zeta)$ which implies that \begin{equation} \label{op1}\overline{[Op(\theta)\,Op(\phi)]}^{\text{w}^{*}}\subseteq Op(\zeta).\end{equation}
On the other hand, let $z\in Op(\zeta).$ From \cite[Theorem 2.1]{Ele-app}, we have that $\Alg{\cl L}=\overline{[Op(\theta)\,X]}^{\text{w}^{*}}$ where $X=\left\{x\in\mathbb{B}(R,K)\mid q^{\perp}\,x\,\theta(q)=0,\,\forall\,q\in\cl M\right\},$ so there exist nets $(u_j)_{j\in J},\,(x_j)_{j\in J}$ such that $u_j\in Op(\theta),\,x_j\in X,\,\forall\,j\in J$ and $$\sum_{r=1}^{k_j}u_{j_r}\,x_{j_r}\stackrel{\text{w}^{*}}{\to} I_{R}$$ which implies that \begin{equation} \label{xxx}\sum_{r=1}^{k_j} u_{j_r}\,x_{j_r}\,z\stackrel{\text{w}^{*}}{\to} z.\end{equation} We observe that \begin{align*}
    X&=\left\{x\in\mathbb{B}(R,K)\mid q^{\perp}\,x\,\theta(q)=0,\,\forall\,q\in\cl M\right\}\\&=\left\{x\in\mathbb{B}(R,K)\mid  \phi(p)^{\perp}\,x\,\theta(\phi(p))=0,\,\forall\,p\in\cl N\right\}\\&=\left\{x\in\mathbb{B}(R,K)\mid  \phi(p)^{\perp}\,x\,\zeta(p)=0,\,\forall\,p\in\cl N\right\}.
\end{align*}
We claim that $x_{j_r}\,z\in Op(\phi)\,,r=1,...,k_j,\,j\in J.$ Indeed, for all $p\in\cl N$ it holds that \begin{align*}
    \phi(p)^{\perp}\,x_{j_r}\,z\,p&=\phi(p)^{\perp}\,x_{j_r}\,(\zeta(p)+\zeta(p)^{\perp})\,z\,p\\&=\phi(p)^{\perp}\,x_{j_r}\,\zeta(p)\,z\,p+\phi(p)^{\perp}\,x_{j_r}\zeta(p)^{\perp}\,z\,p\\&=0
\end{align*}
since $\phi(p)^{\perp}\,x_{j_r}\,\zeta(p)=0$ ($\,x_{j_r}\in X$) and $\zeta(p)^{\perp}\,z\,p=0$ ($\,z\in Op(\zeta)$). Thus, $$z\stackrel{(\ref{xxx})}{=}\lim_{j}\sum_{r=1}^{k_j}u_{j_r}\,x_{j_r}\,z\in \overline{[Op(\theta)\,Op(\phi)]}^{\text{w}^{*}}$$ where the above limit is on the w*-topology, that is \begin{equation}\label{op2}Op(\zeta)\subseteq \overline{[Op(\theta)\,Op(\phi)]}^{\text{w}^{*}}.\end{equation} By (\ref{op1}) and (\ref{op2}) we have that $Op(\zeta)=\overline{[Op(\theta)\,Op(\phi)]}^{\text{w}^{*}}.$

$(ii)\implies (i)$ Suppose now that $Op(\zeta)=\overline{[Op(\theta)\,Op(\phi)]}^{\text{w}^{*}}.$ From the direction $(i)\implies (ii)$ we have that $\overline{[Op(\theta)\,Op(\phi)]}^{\text{w}^*}=Op(\theta\circ \phi),$ so $$Op(\zeta)=Op(\theta\circ \phi).$$ Since $\zeta$ and $\theta\circ \phi$ are continuous onto nest homomorphims from $\cl N$ to $\cl L,$ by Lemma \ref{restriction} we deduce that $$\zeta=\Map{Op(\zeta)}|_{\cl N}=\Map{Op(\theta\circ \phi)}|_{\cl N}=\theta\circ \phi.$$
\end{proof}

\begin{lem}
\label{tensor}
Let $\cl N,\,\cl M,\,\cl L$ be nests with corresponding nest algebras $$\cl A=\Alg{\cl N},\,\cl B=\Alg{\cl M},\,\cl C=\Alg{\cl L}$$ Let also $\phi:\cl N\to \cl M,\,\theta:\cl M\to \cl L$ be continuous onto nest homomorphisms. Then $$Op(\theta)\otimes_{\cl B}^{\sigma h}Op(\phi)\cong \overline{[Op(\theta)\,Op(\phi)]}^{\mathrm{w}^*}$$ as $\cl A$-right $\rm{w}^*$-rigged modules.
\end{lem}

\begin{proof}
The map $Op(\theta)\times Op(\phi)\to Op(\theta)\,Op(\phi),\,(x,u)\mapsto x\,u$ is a separately w*-continuous and completely contractive right $\cl B$-balanced map, so it induces a w*-continuous and completely contractive map $$\rho:Op(\theta)\otimes_{\cl B}^{\sigma h}Op(\phi)\to \overline{[Op(\theta)\,Op(\phi)]}^{\text{w}^{*}}$$ satisfying $\rho(x\otimes_{\cl B} u)=x\,u,\,x\in Op(\theta),\,u\in Op(\phi).$ The map $\rho$ is also a right $\cl A$-module map. Set $Y=\left\{y\mid  q^{\perp}\,y\,\theta(q)=0,\,\forall\,q\in\cl M\right\}.$ From \cite[Theorem 2.1]{Ele-app} we have that $\cl C=\overline{[Op(\theta)\,Y]}^{\text{w}^{*}}.$ Therefore, there exist contractive  nets $(z_t)_{t\in T}\subseteq R^{fin}_{\infty}(Op(\theta))$ and $(y_t)_{t\in T}\subseteq C^{fin}_{\infty}(Y)$ such that $$z_t\,y_t\stackrel{\text{w}^{*}}{\to} 1_{\cl C}$$ where, in general, if $W\subseteq \mathbb{B}(H,K)$ we denote by $R^{fin}_{\infty}(W)$ (resp. $C^{fin}_{\infty}(W)$) the space of operators $(w_1,w_2,...):H^{\infty}\to K$ (resp. $(w_1,w_2,...)^{t}:H\to K^{\infty}$) such that $w_i\in W,\,\forall\,i\in\mathbb{N}$ and also there exists $i_0\in\mathbb{N}$ such that $w_i=0,\,\forall\,i\geq i_0.$ Assume that $x_i\in \cl F,\,u_i\in Op(\phi),\,i=1,...,n,$ where $\cl F$ is the space of all finite rank operators in $Op(\theta).$ Since $x_i\in\cl F$ we have that $z_t\,y_t\,x_i\stackrel{||\cdot||}{\to}x_i.$ By the fact that the Haagerup tensor product is contractive as a bilinear map, for every $\epsilon>0$ there exists $t\in T$ such that \begin{equation} \label{ineq} \nor{\sum_{i=1}^n x_i\otimes_{\cl B}u_i}-\epsilon\leq \nor{\sum_{i=1}^n z_t\,y_t\,x_i\otimes_{\cl B}u_i}.\end{equation} Let $x\in\cl F\subseteq Op(\theta)$ and $y\in Y.$ It holds that $\theta(q)^{\perp}\,x\,q=0,$ so it follows that $x\,q=\theta(q)\,x\,q,\,\forall\,q\in\cl M$ and as a consequence $$q^{\perp}\,y\,x\,q=q^{\perp}\,y\,\theta(q)\,x\,q=0,\,\forall\,q\in\cl M\, (\text{since}\,\,y\in Y),$$  which means that $Y\,\cl F\subseteq \cl B.$ Using (\ref{ineq}) we get that $$\nor{\sum_{i=1}^n x_i\otimes_{\cl B}u_i}-\epsilon\leq \nor{z_t\otimes_{\cl B}y_t\,\sum_{i=1}^n x_i\,u_i}\leq \nor{\sum_{i=1}^n x_i\,u_i}=\nor{\rho\left(\sum_{i=1}^n x_i\otimes_{\cl B}u_i\right)}.$$ But since $\epsilon$ is arbitrary we have that $$\nor{\sum_{i=1}^n x_i\otimes_{\cl B}u_i}\leq \nor{\rho\left(\sum_{i=1}^n x_i\otimes_{\cl B}u_i\right)}$$ which implies that the restriction of $\rho$ to $$\cl F\otimes_{\cl B}^h Op(\phi)=\overline{[x\otimes_{\cl B}u\mid x\in\cl F,\,u\in Op(\phi)]}^{||\cdot||}$$ is an isometry. We will prove that $\rho$ is an isometry in $Op(\theta)\otimes_{\cl B}^{\sigma h}Op(\phi).$ To this end, let $z\in Op(\theta)\otimes_{\cl B}^{\sigma h}Op(\phi).$ Thus, there exists a net $(z_{\ell})_{\ell \in I}$ in $\cl F\otimes_{\cl B}^h Op(\phi)$ such that $z_{\ell}\stackrel{\text{w}^{*}}{\to}z.$ Let $f\in Ball(\cl C),\,g\in Ball(\cl A)$ be finite rank operators. Then $f\,z_{\ell}\,g\stackrel{\text{w}^{*}}{\to}f\,z\,g$ and since $\rho$ is w*-continuous we have that $ \rho(f\,z_{\ell}\,g)\stackrel{\text{w}^{*}}{\to}\rho(f\,z\,g)$ which implies that $f\,\rho(z_{\ell})\,g\stackrel{\text{w}^{*}}{\to}\rho(f\,z\,g).$ Using the fact that $f,\,g$ are finite rank operators we deduce that $$||f\,z_{\ell}\,g||=||\rho(f\,z_{\ell}\,g)||=||f\,\rho(z_{\ell})\,g||\to ||\rho(f\,z\,g)||.$$ For every $\epsilon>0$ there exists $\ell_{0}\in I$ such that $||f\,z_{\ell}\,g||<||\rho(f\,z\,g)||+\epsilon,\,\forall\,\ell\geq \ell_{0}$ so, $||f\,z_{\ell}\,g||<||\rho(z)||+\epsilon,\,\forall\,\ell\geq \ell_{0}.$ We have that $f\,z_{\ell}\,g\stackrel{\text{w}^{*}}{\to}f\,z\,g$ which implies that $||f\,z\,g||\leq \sup\,\left\{||f\,z_{\ell}\,g||: \ell\geq \ell_{0}\right\}.$ Thus, $||f\,z\,g||<||\rho(z)||+\epsilon.$ Since $\epsilon$ is arbitrary we get \begin{equation} \label{ast} ||f\,z\,g||\leq ||\rho(z)||.\end{equation} By \cite[Corollary 3.13]{Dav} there exist contractive nets of finite rank operators $(f_{j})_{j\in J}\subseteq \cl C$ and $(g_{j})_{j\in J}\subseteq \cl A$ such that $$f_{j}\stackrel{\text{w}^{*}}{\to}1_{\cl C},\,g_{j}\stackrel{\text{w}^{*}}{\to}1_{\cl A}.$$ Thus, $f_{j}\,z\,g_{j}\stackrel{\text{w}^{*}}{\to} z$ and according to (\ref{ast}) it holds that $$||z||\leq \sup_{j\in J}||f_j\,z\,g_j||\leq ||\rho(z)||$$ which implies that $\rho$ is an isometry. Similarly we can prove that $\rho$ is a complete isometry and as a consequence (Krein-Smulian) $\rho$ is onto $\overline{[Op(\theta)\,Op(\phi)]}^{\text{w}^{*}}.$

\end{proof}

The following theorem describes how to construct Morita equivalent right w*-rigged modules.

\begin{thm}
\label{morita}
Let $\cl N_1,\,\cl N_2,\,\cl M_1,\,\cl M_2$ be nests, $\,\cl A=\Alg{\cl N_1}\subseteq \mathbb{B}(H_1),\,\cl B=\Alg{\cl N_2}\subseteq \mathbb{B}(H_2),$ $\phi:\cl N_1\to \cl M_1,\,\psi:\cl N_2\to \cl M_2$ be continuous onto nest homomorphisms and $E=Op(\phi),\,F=Op(\psi).$ If $\chi:\cl N_1\to \cl N_2$ is a nest isomorphism such that $E\cong Op(\psi\circ \chi)$ as $\cl A$-right $\mathrm{w}^*$-rigged modules, then $E$ and $F$ are Morita equivalent.

\end{thm}

\begin{proof}

Let $\chi:\cl N_1\to \cl N_2$ be a nest isomorphism such that $E\cong Op(\zeta)$ as $\cl A$-right w*-rigged modules, where $\zeta=\psi\circ \chi.$ By Lemma \ref{composition} we get $$E\cong Op(\zeta)=\overline{[Op(\psi)\,Op(\chi)]}^{\text{w}^{*}}=\overline{[F\,Op(\chi)]}^{\text{w}^{*}}.$$ We define the $\cl B$-right w*-rigged module $Y=Op(\chi^{-1}).$ Let $$X=\left\{x\in\mathbb{B}(H_1,H_2)\mid q^{\perp}\,x\,\chi^{-1}(q)=0,\,\forall\,q\in\cl N_2\right\}=Op(\chi).$$ By \cite[Theorem 2.9]{Ele-mor} we have that $\cl A\cong Y\otimes_{\cl B}^{\sigma h}X,\,\cl B\cong X\otimes_{\cl A}^{\sigma h}Y.$ According to Remark \ref{for-ref} we have that $\tilde{Y}\cong X=Op(\chi)$ and thus $$\cl A\cong Y\otimes_{\cl B}^{\sigma h}\tilde{Y},\,\,\cl B\cong \tilde{Y}\otimes_{\cl A}^{\sigma h}Y.$$ Furthermore by Lemma \ref{tensor} we have that $$E\cong Op(\zeta)\cong \overline{[Op(\psi)\,Op(\chi)]}^{\text{w}^{*}}\cong Op(\psi)\otimes_{\cl B}^{\sigma h}Op(\chi)\cong F\otimes_{\cl B}^{\sigma h}\tilde{Y}.$$
Therefore, $E$ and $F$ are Morita equivalent.

\end{proof}

\begin{ex}
Let $\cl N_1,\,\cl N_2$ be isomorphic nests acting on the Hilbert spaces $H_1,\,H_2$ respectively. Let $\cl A=\Alg{\cl N_1}\subseteq \mathbb{B}(H_1)$ and $\cl B=\Alg{\cl N_2}\subseteq \mathbb{B}(H_2)$ be the corresponding nest algebras and let $\chi:\cl N_1\to \cl N_2$ be a nest isomorphism. We consider the nest $$\cl N_2\oplus \cl N_1=\left\{\chi(p)\oplus p\mid p\in\cl N_1\right\}\subseteq \mathbb{B}(H_2\oplus H_1)$$ and we define the continuous onto nest homomorphism $$\psi:\cl N_2\to \cl N_2\oplus \cl N_1,\,\psi(q)=(q,\chi^{-1}(q))$$ If $F=Op(\psi)$ then by Theorem \ref{morita} the right w*-rigged modules $E=Op(\psi\circ \chi)$ and $F$ are Morita equivalent.

\end{ex}

\begin{thm}

\label{d-morita}
Let $\cl N_1,\,\cl N_2,\,\cl M_1,\,\cl M_2$ be nests and $\,\cl A=\Alg{\cl N_1},\,\cl B=\Alg{\cl N_2}.$ Let also $\,\phi:\cl N_1\to \cl M_1,\,\psi:\cl N_2\to \cl M_2$ be continuous onto nest homomorphisms and $E=Op(\phi),\, F=Op(\psi)$. If $\chi:\cl N_1\to \cl N_2$ is a nest isomorphism which can be extended to a $*$-isomorphism $\chi':\cl N_1^{\prime \prime}\to \cl N_2^{\prime \prime}$ such that $E\cong Op(\psi\circ \chi)$ as $\cl A$-right $\mathrm{w}^*$-rigged modules, then $E$ and $F$ are strongly projectively Morita equivalent.

\end{thm}

\begin{proof}

Let $\chi:\cl N_1\to \cl N_2$ be a nest isomorphism which can be extended to a $*$-isomorphism $\chi':\cl N_1^{\prime \prime}\to \cl N_2^{\prime \prime}$ such that $E\cong Op(\psi\circ \chi)$ as $\cl A$-right w*-rigged modules. By the fact that $(\chi^{\prime})^{-1}:\cl N_2^{\prime \prime}\to \cl N_1^{\prime \prime}$ is a $*$-isomorphism such that $(\chi^{\prime})^{-1}(\cl N_2)=\chi^{-1}(\cl N_2)=\cl N_1,$ we have that the space $Y=Op(\chi^{-1})$ is a concrete strongly projectively right w*-rigged module over $\cl B,$ see Corollary \ref{proj-nest}. Then $\tilde{Y}\cong Op(\chi)$ (Remark \ref{for-ref}) and $Y$ is a bimodule of w*-Morita equivalence between the algebras $\cl A$ and $\cl B.$ By Lemma \ref{composition}, we have that $Op(\psi\circ \chi)=\overline{[Op(\psi)\,Op(\chi)]}^{\text{w}^{*}}$ and from Lemma \ref{tensor} we deduce that $$E\cong Op(\psi\circ \chi)=\overline{[Op(\psi)\,Op(\chi)]}^{\text{w}^{*}}\cong Op(\psi)\otimes_{\cl B}^{\sigma h}Op(\chi)\cong F\otimes_{\cl B}^{\sigma h}\tilde{Y}.$$
Therefore, $E$ and $F$ are strongly projectively Morita equivalent.
\end{proof}

\textbf{Acknowledgements}: I would like to express my appreciation to my supervising professor G.K. Eleftherakis for his helpful suggestions during the preparation of this paper. I am also grateful to the referee and the reviewers for carefully reading the paper and making significant improvements.

\end{document}